\documentclass{amsart}
\usepackage{palatino, mathpazo}
\usepackage{mathrsfs}
\usepackage{amsfonts, amsthm, amsmath,amssymb}
\usepackage{enumerate}
\usepackage[all]{xy}
\usepackage{hyperref}
\newtheorem{thm}{Theorem}[section]
\newtheorem{defn}[thm]{Definition}

\newtheorem{corollary}[thm]{Corollary}
\newtheorem{lemma}[thm]{Lemma}
\newtheorem*{mlemma}{Main Lemma}

\theoremstyle{remark}
\newtheorem{remark}[thm]{Remark}

\newcommand{\bC}{\mathbb{C}}

\newcommand{\fD}{\mathfrak{D}}
\newcommand{\fM}{\mathfrak{M}}
\newcommand{\fP}{\mathfrak{P}}

\newcommand{\sM}{\mathscr{M}}
\newcommand{\fs}{\mathbf{LogSt}}
\newcommand{\lsch}{\mathbf{LogSch}}
\newcommand{\sch}{\mathbf{Sch}}
\newcommand{\St}{\mathbf{St}}
\newcommand{\gpd}{\mathbf{Grpd}}
\newcommand{\M}{M} 

\newcommand{\mbar}{\overline{\mathcal{M}}}
\newcommand{\vir}{\mathrm{vir}}

\newcommand{\pr}{\mathrm{pr}}

\DeclareMathOperator{\spec}{Spec}
\DeclareMathOperator{\Tor}{\mathcal{T}\!or}
\DeclareMathOperator{\Log}{\mathcal{L}\!og}

\title{A product formula for log Gromov--Witten invariants}

\author[Y.-P. \textsc{Lee}]{Yuan-Pin \textsc{Lee}}
\address{Department of Mathematics, University of Utah,
Salt Lake City, Utah 84112-0090, U.S.A.}
\email{yplee@math.utah.edu}

\author[F. \textsc{Qu}]{Feng \textsc{Qu}}
\address{
Beijing International Center for Mathematical Research, 
Peking University, No.~5 Yiheyuan Road, 
Beijing, 100871, China.}
\email{fengquest@gmail.com}

\begin{document}

\begin{abstract}
The purpose of this short article is to prove a product formula relating the 
log Gromov--Witten invariants of $V \times  W$ with those of $V$ and $W$
in the case the log structure on $V$ is trivial.
\end{abstract}

\maketitle

\setcounter{section}{-1}

\section{Introduction}
Product formulas in the literature of Gromov--Witten (GW) theory started 
with \cite{KM} for the genus zero Gromov--Witten invariants.
It was soon generalized in \cite{B} to (absolute) GW invariants
in any genus.
There is also an orbifold version in \cite{AJT}.
In this paper, we extend the product formula to the setting of 
\emph{relative GW theory} or more generally \emph{log GW theory}.

The absolute Gromov--Witten theory studies the intersection theory
on the moduli stacks of stable maps $\mbar_{g,n}(X, \beta)$
from an $n$-pointed curve with arithmetic genus $g$ 
to a fixed nonsingular projective variety $X$ with a fixed degree $\beta$ 
in the Mori cone $NE(X)$ of effective curves in $X$ 
\[
 f\colon  (C, x_1, x_2, \ldots, x_n) \to X, \, \text{such that }\,
 f_*([C]) = \beta.
\]
For notational convenience, we denote such a class by $[f]$.
Intuitively, GW invariants count the numbers of curves passing through
$n$ fixed cycles $\alpha_1, \ldots, \alpha_n$ in $X$ 
with the above given conditions.
To put it on a mathematically sound setting, one defines the 
invariants as intersection numbers on $\mbar_{g,n}(X, \beta)$
in the following way.
By the functorial properties of the moduli stacks, there are the
\emph{evaluation morphisms}
\[
 ev_i \colon   \mbar_{g,n}(X, \beta) \to X, \, i= 1, \ldots, n,
\]
and \emph{stabilization morphism}
\[
 \pi \colon  \mbar_{g,n}(X, \beta) \to \mbar_{g,n},
\]
where $\mbar_{g,n}$ is the moduli stack of stable genus $g$,
$n$-pointed curves and
\[
  ev_i ([f]) = f(x_i) \in X, \quad 
  \pi ([f]) = [(\overline{C}, \overline{x}_1, \overline{x}_2, \ldots, 
  \overline{x}_n)] \in \mbar_{g,n},
\]
where $\overline{C}$ is the stabilization of the source curve $C$.
The GW invariants can be defined as
\[
 \int_{[\mbar_{g,n}(X, \beta)]^{vir}} \pi^*(\gamma) \prod_i ev_i^*(\alpha_i),
\]
where $[\mbar_{g,n}(X, \beta)]^{vir}$ is the \emph{virtual fundamental class}.
One can rephrase these invariants in terms of the cohomological field theory
\[
 R^X_{g,n,\beta} \colon  H^*(X)^{\otimes n} \to H^*(\mbar_{g,n})
\]
via
\[
 \int_{[\mbar_{g,n}(X, \beta)]^{vir}} \pi^*(\gamma) \prod_i ev_i^*(\alpha_i)
 = \int_{\mbar_{g,n}} \gamma . R^X_{g,n,\beta} ( \prod_i \alpha_i ) . 
\]
The details can be found in \cite{B} and references therein.

Intuitively, the \emph{relative} invariants 
as defined in \cite{LR} and \cite{jL} can be considered as refined counting.
Let $(X, D)$ be a pair consisting of a nonsingular projective variety $X$ and a smooth
divisor $D$ in $X$.
If the curve $C$ does not lie in $D$, it intersects with $D$ at $\rho$ points
with multiplicities $\mu_1, \ldots, \mu_{\rho}$ such that
\[
 \sum_j \mu_j = \int_\beta [D]. 
\]
The refined counting is to fix the profile $(\mu_1, \ldots, \mu_{\rho})$ and
constraint the $\rho$ points to lie in chosen cycles $\{ \delta_j \}$ in $D$.
Similarly, one can define the relative invariants as intersection
numbers on relative moduli stacks as above.
See \cite{jL} and \cite{LR} for details.
There is also a similar reformulation in terms of cohomological field theory.
See Section~\ref{s:3}.

The divisor $D$ in $X$ gives rise to a divisorial log structure on X.
Recently relative invariants have been generalize to the setting of \emph{log} GW (\cite{AC,C,GS}).
See Section~\ref{s:1} for a brief summary of log geometry and log GW theory.

In a sense, the study of Gromov--Witten theory is the study
of the virtual fundamental classes.
The \emph{product formula} is a statement that 
\emph{GW invariants of $V$ and $W$ determine the invariants of $V \times W$}. 
See Equation~\eqref{e:3}.
It can be written in the form of certain
functorial properties of virtual classes.
In \cite{B}, K.~Behrend proves a product formula for absolute
GW invariants by first establishing a corresponding functorial property
of virtual fundamental classes.
In this note, we approach the product formula in the relative and log
settings in a similar way.
The \emph{main results} are Theorem~\ref{t:main} and Corollary~\ref{c:2.2}, 
which \emph{expresses the log/relative invariants of $X \times (Y,D)$ by
invariants of $X$ and of $(Y,D)$}.
The logarithmic approach to relative GW invariants of Abramovich-Chen and Gross-Siebert (\cite{AC,C, GS}) avoids the expanded degenerations of Li (\cite{jL}) and allows us to adapt
Behrend's original proof, but it also presents new technical difficulties.
We were not able to prove the product formula in the 
general log setting and we have to assume the log structure on one of the factors
to be trivial. For the general case, see Section \ref{s3} for some initial attempts and speculations.


This product formula could be useful in the study of Gromov--Witten theory, even when no explicit product  or log geometry is involved in the statement. For example, it plays a role in proving the crepant transformation conjecture for ordinary flops with non-split vector bundles in \cite{LLQW}. There the degeneration technique is extensively employed and the product 
formula is applied to treat fiber integrals which naturally occur in the degeneration process.


\section{Preliminaries} \label{s:1}

\subsection{Log geometry}
We work over the base log scheme $\spec \bC$ with the trivial log structure.
We refer to \cite[Sections~1-4]{K} for general background on
log structures on schemes, and \cite[Section~5]{O2} for log stacks.
For the reader's convenience, we recall
the basic properties we need about 
Olsson's Log stack and 
saturated morphisms.

\subsubsection{$\Tor_X$}
Denote by $\fs$ the category of fine saturated (fs) log algebraic stacks, and $\St$ the category of algebraic stacks.
\footnote{In a higher categorical sense. 
In other words, $\fs$ and $\St$ are 2-categories, or (2,1) categories.}

For a fs log stack $X$, $\Tor_X$ is introduced in \cite[Remark 5.26]{O2} to parametrize all fs log schemes over $X$. It follows from the definition of $\Tor_X$ that a map $S \to X$ in $\fs$ factors 
as 
$
 S \to \Tor_X \to X.
$
We can view $X$ as an open substack of $\Tor_X$ parametrizing strict maps $S \to X$.
 Note that $\Tor_X$ has a natural fs log structure so that the factorization $S \to \Tor_X$ above is strict, and $\Tor_X \to X$ is log \'{e}tale.

It is easy to check when $X$ is log smooth, it is open dense in $\Tor_X$. 
%
In particular if $X$ is log smooth and irreducible, $\Tor_X$ is irreducible.

\begin{remark}
Let $\sch$ be the category of schemes over $\bC$, $\gpd$ the category of (small) groupoids.
Recall a stack is a functor 
$\sch \to \gpd$ that satisfies certain 'sheaf' condition. A log stack is a stack with a log struture,
and a log structure can be understood as a map from the stack to $\Log_\bC$.

It is more natural to describe a fs log moduli stack $X$ by its functor of points than
specifying its underlying stack and log structure. 
As the Yoneda embedding for the (2,1) category $\fs$ is fully faithful, an object $X$ in $\fs$ 
is uniquely determined by:
\[
\hom_{\fs}(-, X)\colon \lsch \to \gpd.
\]
Here the stack condition for $X$ allows us to restrict $\hom_{\fs}(-, X)$ from $\fs$
to its full subcategory $\lsch$ of fs log schemes. 

Given a log moduli functor $\mathbb{X} \colon \lsch \to \gpd$ such that it is isomorphic to 
$\hom_{\fs}(-, X)$ for some fs log stack $X$, one 
might recover $X$ by considering the stack $\Tor_\mathbb{X}\colon \sch \to \gpd$ which corresponds to 
$\Tor_X$, then the underlying stack $\underline{X}$ of $X$ is the substack of 
$\Tor_\mathbb{X}$ satisfying a minimal condition, and the log structure of $X$
is the restriction of the natural log structure on $\Tor_\mathbb{X}$. See \cite{G1} for details.

\end{remark}

\subsubsection{Saturated morphisms}
\begin{defn}
Let $P, Q$ be saturated monoids, a \emph{map $P \to Q$ is saturated}, 
if it is integral and the push out of
\[
\xymatrix{
P \ar[r]\ar[d] & Q \\
R  
}
\]
is saturated when $R$ is saturated.
\end{defn}

\begin{defn}
Let $(X, \M_X), (Y, \M_Y) $ be fs log schemes, 
a map $f \colon  (X, \M_X) \to (Y, \M_Y)$ is called \emph{saturated},
if for any $x \in X, y=f(x)$, the induced map between characteristics 
$\overline{\M}_{Y,\bar{y}} \to \overline{\M}_{X, \bar{x}}$ is saturated.
\end{defn}

\begin{remark}
Given the above definition, a saturated morphism between fs log stacks can be 
defined locally with respect to the lisse-\'{e}tale topology.
\end{remark}

It follows from the definition that  saturated morphisms satisfy the following properties:
\begin{itemize}
\item They are stable under composition and base change in $\fs$.
\item
For a cartesian square
\[
\xymatrix{
(W,\M_W) \ar[r]\ar[d]   & (X,\M_X) \ar[d]^f\\
(Z, \M_Z)  \ar[r]            &(Y, \M_Y)\\
}
\] 
in $\fs$, when  $f$ is saturated, the underlying diagram of stacks
\[
\xymatrix{
W \ar[r]\ar[d]  &  X\ar[d] \\
Z  \ar[r]           & Y
}
\]
is cartesian in $\St$.
\end{itemize}

\subsection{Log Gromov--Witten theory}

\subsubsection{Log curves}
Let $\mbar_{g,n}$ (resp.\ $\mathfrak{M}_{g,n}$) be the algebraic stack of stable 
(resp.\ prestable) curves of genus $g$ with $n$ marked points.  
The log structure of $\mbar_{g,n}$ is the divisorial log structure associated 
with its boundary consisting of singular curves.  
The log structure of $\mathfrak{M}_{g,n}$ is determined from the  smooth chart 
$\sqcup_{m}\mbar_{g,n+m} \to \mathfrak{M}_{g,n}$, where the map from 
$\mbar_{g,n+m} \to \mathfrak{M}_{g,n}$ is forgetting the extra m marked points 
without stabilizing.   We view $\mbar_{g,n}$ and $\mathfrak{M}_{g,n}$ as log stacks from now on using the same notation.

Note that the forgetful map $\mbar_{g,n+1} \to \mbar_{g,n}$ is saturated, 
as it can be identified as the universal log curve. This implies the stabilization map
$\mathfrak{M}_{g,n} \to \mbar_{g,n}$ is saturated. 

\subsubsection{Log stable maps}
For a projective log smooth scheme $V$, let $\sM_{g,n}(V)$ be the fs log stack 
parameterizing stable log maps from log curves of genus $g$ with $n$ marked 
points (see \cite{AC,ACMW,C,GS} for more details).  
For a fs log scheme $S$,  a log map from $S$ to $\sM_{g,n}(V)$ corresponds to a 
stable log  map:
\begin{equation} \label{e:1}
\xymatrix{
C \ar[r]\ar[d] &  V\\
S &.
}
\end{equation}
We note that the underlying map of the diagram is a usual stable map.

A log map $ V \to W$ induces a stabilization map $\sM_{g,n}(V) \to \sM_{g,n}(W)$.
Given $S \to \sM_{g,n}(V)$ which corresponds to a diagram \eqref{e:1}, 
the composition $S \to \sM_{g,n}(V) \to \sM_{g,n}(W) $ corresponds to 
\[
\xymatrix{
\overline{C} \ar[r]\ar[d] &  W\\
S &.
}
\] 
where the underlying map of $\overline{C} \to W$ is the stabilization of 
the underlying map $ C \to V \to W$, and the log structure on $\overline{C}$
is  the push forward of the log structure on $C$ with respect to the partial 
stabilization $C \to \overline{C}$. 
(See \cite[appendix B]{AMW}.)
\subsubsection{Perfect obstruction theories}
We have a natural log map $\sM_{g,n}(V) \to \mathfrak{M}_{g,n}$ forgetting the map to $V$.
As log deformations for $\sM_{g,n}(V) \to \mathfrak{M}_{g,n}$ is the same as 
deformations for  the underlying map
of $\sM_{g,n}(V) \to \Tor_{\fM_{g,n}}$, a (relative) perfect obstruction theory for the log map 
$\sM_{g,n}(V) \to \mathfrak{M}_{g,n}$
is by definition a perfect obstruction theory for the underlying map of $\sM_{g,n}(V) \to \Tor_{\fM_{g,n}}$.


A perfect obstruction theory for $\sM_{g,n}(V) \to \fM_{g,n}$
is defined in \cite[Section~5]{GS},
tangent space and obstruction space at $[f\colon C \to V] \in \sM_{g,n}(V)$
are given by $H^0(f^*T_V)$ and $H^1(f^*T_V)$ respectively, where
$T_V$ is the log tangent bundle of V.

If we have a factorization $\sM_{g,n}(V) \to \mathfrak{N} \to \mathfrak{M}_{g,n}$, 
and $\mathfrak{N} \to \mathfrak{M}_{g,n}$ is log \'{e}tale,
a perfect obstruction theory for $\sM_{g,n}(V) \to \fM_{g,n}$ 
induces a perfect obstruction theory for $\sM_{g,n}(V) \to \mathfrak{N}$ since
$\Tor_{\mathfrak{N}} \to \Tor_{\fM_{g,n}}$ is \'etale.



\section{Product formula in terms of virtual classes} \label{s:2}

\subsection{Setup}
We start by introducing relevant commutative diagrams, which are the log enhancement
of those in \cite{B}.

We define a fs log stack $\fD$ 
by its functor of points. Given any fs log scheme $S$, a map $S \to \fD$ corresponds
to the data consisting of $n$-pointed
prestable log curves $C, C', C''$ of genus $g$ over $S$,
together with partial stabilizations $p'\colon C \to C'$ and $p''\colon C \to C''$, 
such that no component of $C$ is contracted by both $p',p''$. Note that as log curves over $S$, the log structure of $C'$ resp. $C''$
is given by the pushforward of the log structure of $C$ along $p'$ resp. $p''$.

Define $e\colon \fD \to  \mathfrak{M}_{g,n}$ as the forgetful morphism taking
$(p',p'')$ to $C$.
It is proved in \cite[Lemma 4]{B} that the underlying map of $e$ is \'{e}tale.
As $e$ is strict, it is log \'{e}tale.

We claim that the following commutative diagram is cartesian
in $\fs$,
\begin{equation} \label{e:2}
\xymatrix{
\sM_{g,n}(V \times W)  \ar[r]\ar[d] & \sM_{g,n}(V) \times
\sM_{g,n}(W)  \ar[d]\\
\fD \ar[r] &\fM_{g,n} \times \fM_{g,n}
}
\end{equation}
Here the top horizontal arrow is determined by stabilization maps
$\sM_{g,n}(V \times W) \to \sM_{g,n}(V)$,
$\sM_{g,n}(V \times W) \to \sM_{g,n}(W)$
induced from the projections  $\pr_1\colon V \times W \to V$,
$\pr_2\colon V \times W \to W$.
The left vertical arrow is defined by retaining the curve together with  
its partial stabilizations with respect to $\pr_1,\pr_2$.
 
The reason for the diagram being cartesian is essentially the same as that 
in \cite[Proposition 5]{B}. For a fs log scheme S,
a commutative diagram  
\[
\xymatrix{
 S \ar[r]\ar[d]        &\sM_{g,n}(V) \times \sM_{g,n}(W) \ar[d] \\
\fD \ar[r]               &\fM_{g,n} \times \fM_{g,n}\\
}
\]
corresponds to  stable log maps $C' \to V, C'' \to W$ over $S$, 
together with stabilization between log curves $C \to C', C \to C''$.
\footnote{Note that $\sM_{g,n}(V)(S)\to \fM_{g,n}(S)$ and $\sM_{g,n}(W)(S)\to \fM_{g,n}(S)$ are iso-fibrations of groupoids.}
This recovers a stable log map $C \to C'\times C''  \to V \times W$.
Indeed, $\fD$ is constructed to make \eqref{e:2} cartesian.
 
We then extend the above cartesian diagram 
in $\fs$ to
\[
\xymatrix{
\sM_{g,n}(V \times W)  \ar[r]^-h \ar[d]^c & P\ar[r]\ar[d]  &  \sM_{g,n}(V) \times \sM_{g,n}(W) \ar[d]^a\\
\fD \ar[r]^l\ar[d]^e                                   &  \fP \ar[r]^-\phi \ar[d]& \fM_{g,n} \times \fM_{g,n} \ar[d] \\
\fM_{g,n} &   \mbar_{g,n} \ar[r]^-\Delta & \mbar_{g,n}  \times \mbar_{g,n} \, .
}
\]
Here all squares are constructed by taking fiber product. If $V$ and $W$ have trivial log structures, this reduces to Diagram (2) of \cite{B}. 
 
\begin{defn}
We say that the product formula (of virtual fundamental classes) holds for $V$ and $W$ if
\begin{equation} \label{e:3}
h_*([\sM_{g,n}(V \times W)]^{\vir})
=\Delta^!([\sM_{g,n}(V)]^{\vir} \times [\sM_{g,n}(W)]^{\vir}).
\end{equation}
\end{defn} 
 
In \cite{B}, Behrend showed that the
product formula holds for $V$ and $W$ when $V$ and $W$ are smooth projective schemes, or log smooth schemes with
trivial log structures, our goal of this paper is to
extend his result to the case when $W$ has nontrivial log structure.




\subsection{Product formula in log GW theory} \label{QQQ}
When $W$ has nontrivial log structure, %
we factor 
\[
 a\colon  \sM_{g,n}(V) \times \sM_{g,n}(W) \to \fM_{g,n} \times  \fM_{g,n} 
\]
into 
\[
 \sM_{g,n}(V) \times \sM_{g,n}(W)  \overset{a'}{\to} \fM_{g,n} \times \Tor_{\fM_{g,n}}
 \to \fM_{g,n} \times  \fM_{g,n}.
\]
We then have the following commutative diagram in which 
all squares are cartesian in $\fs$,
\begin{equation} \label{e:4}
\xymatrix{
\sM_{g,n}(V \times W)  \ar[r]^-h \ar[d]^{c'} & P\ar[r]\ar[d]  &  \sM_{g,n}(V) \times \sM_{g,n}(W) \ar[d]^{a'}\\
\fD' \ar[r]^{l'}\ar[d]                                 &  \fP' \ar[r]^-{\phi'} \ar[d]& \fM_{g,n} \times \Tor_{\fM_{g,n}} \ar[d] \\
\fD \ar[r]^l\ar[d]                                   &  \fP \ar[r]^-\phi \ar[d]& \fM_{g,n} \times \fM_{g,n} \ar[d] \\
\fM_{g,n} &   \mbar_{g,n} \ar[r]^-\Delta & \mbar_{g,n}  \times \mbar_{g,n} \, .
}
\end{equation}

\begin{mlemma} \label{l:main}
 \begin{enumerate}[(I)]
\item The underlying square diagrams in $\St$  are all cartesian. 
\item The relative perfect obstruction theories for $a'$ and $c'$ are compatible.
\item $\fD', \fP'$ are irreducible, and $l'$ is of degree 1.
\item $\phi'$ is a l.c.i. and is compatible with $\Delta$ in the sense that  the cotangent complex 
$\mathcal{L}_{\Delta}$ pulls back
to $\mathcal{L}_{\phi'}$.
\end{enumerate}
\end{mlemma}

\begin{thm} \label{t:main}
Using the notations in \eqref{e:4}, we have
\[
 h_*([\sM_{g,n}(V \times W)]^{\vir})
=\Delta^!([\sM_{g,n}(V)]^{\vir} \times [\sM_{g,n}(W)]^{\vir})
\]
in the log GW setting, assuming the trivial log structure on $V$.
\end{thm}

\begin{proof}
We have virtual pullbacks $c'^!, a'^!, \phi'^!,$ and $\Delta^!$. (See \cite[Section 3.1]{M1})

For the upper left square of diagram \eqref{e:4},
we have $$h_*([\sM_{g,n}(V \times W)]^{\vir})=a'^![\fP'],$$
by (II), (III), and Costello's pushforward theorem \cite[Theorem 5.0.1]{Co}.

For the upper right square of diagram \eqref{e:4}, 
note that $\phi'^!([\fM_{g,n}] \times [\Tor_{\fM_{g,n}}])$  equals $[\fP']$, 
and $a'^!([\fM_{g,n}] \times [\Tor_{\fM_{g,n}}])$ is  $[\sM_{g,n}(V)]^{\vir} \times [\sM_{g,n}(W)]^{\vir}$.
By \cite[Theorem 4.3]{M1}, we know that $a'^!\phi'^!=\phi'^!a'^!$, so 
$$a'^![\fP']=\phi'^!([\sM_{g,n}(V)]^{\vir} \times [\sM_{g,n}(W)]^{\vir}).$$ 

Now (IV) gives 
$$\phi'^!([\sM_{g,n}(V)]^{\vir} \times [\sM_{g,n}(W)]^{\vir}]) = 
\Delta^!([\sM_{g,n}(V)]^{\vir} \times [\sM_{g,n}(W)]^{\vir}),$$ 
and we can complete the proof by combining these equations.
\end{proof}

\subsection{Proof of Main Lemma}
\subsubsection*{(I)} 
As $a'$ is strict, the first row of squares are cartesian.
Since $\fM_{g,n} \to \mbar_{g,n}$ is saturated, the bottom square is cartesian.

For the second row, to prove the two squares are cartesian is the same 
as showing the squares in the following diagram are cartesian in $\St$.
\begin{equation} \label{e:5}
\xymatrix{
\fD' \ar[r]^{l'} \ar[d]  & \fP' \ar[r]^-{\pr_2\circ\phi'} \ar[d] &   \Tor_{\fM_{g,n}} \ar[d]\\
\fD \ar[r]^l               &  \fP \ar[r]^-{\pr_2\circ\phi}  \ar[r] & \fM_{g,n} .
}
\end{equation}
We claim that $l \circ \pr_2\circ\phi$ and $\pr_2\circ\phi$ are both saturated.
$l \circ \pr_2\circ\phi$ is saturated because a partial stabilization map locally  
is given by forgetting marked points.
(See the proof of \cite[Proposition 3]{B}. 
What we need is that Diagram (4) there to be commutative.)
$\pr_2\circ\phi$ is saturated as it is the base change by
$\fM_{g.n} \to \mbar_{g,n}$
\[
\xymatrix{
\fP \ar[r]^-{\pr_2\circ\phi} \ar[d]  & \fM_{g,n} \ar[d]\\
\fM_{g.n}\ar[r]                          & \mbar_{g,n} .\\
}
\]
Thus, both squares in \eqref{e:5} are cartesian.

\subsubsection*{(II)} 
We remark that  as $\Tor_{\fM_{g,n}} \to \mathfrak{M}_{g,n}$ is log \'{e}tale, 
$\fD'$ is log \'{e}tale over $\fD$.
Thus the relative perfect obstruction theory for the log map
$\sM_{g,n}(V \times W) \to \mathfrak{M}_{g,n}$ can be viewed as a 
relative perfect obstruction theory for the underlying map of $c'$. 
Compatibility check is the same as in \cite[Propsition 6]{B}.

\subsubsection*{(III)} 
Start with Diagram \eqref{e:5}.
The open embedding 
$\fM_{g,n} \to \Tor_{\fM_{g,n}}$ induces 
\begin{equation} \label{e:6}
\xymatrix{
\fD \ar[r]^l    \ar[d]   & \fP \ar[r]^-{\pr_2\circ\phi}\ar[d]   & \fM_{g,n} \ar[d]\\
\fD' \ar[r]^{l'} \ar[d]   &\fP'  \ar[r]^-{\pr_2\circ\phi'} \ar[d] & \Tor_{\fM_{g,n}} \ar[d]\\
\fD \ar[r]^l                & \fP \ar[r]^-{\pr_2\circ\phi}  & \fM_{g,n} \, .\\
}
\end{equation}
$\mathfrak{M}_{g,n}$ being log smooth, it is open and dense in $\Tor_{\fM_{g,n}}$.
Note that both
$\pr_2\circ\phi, l\circ\pr_2\circ\phi$ are flat surjective, we conclude 
$\fD$(resp. $\fP$) are open dense substacks of $\fD'$ (resp.\ $\fP'$). 
(III) then follows from properties of $\fD, \fP$ and $l$. (See \cite[Proposition 3]{B}).

\subsubsection*{(IV)} 
We factor $\fM_{g,n} \times \Tor_{\fM_{g,n}} \to \mbar_{g,n} \times \mbar_{g,n}$ 
into
\[
\fM_{g,n} \times \Tor_{\fM_{g,n}} \to \mbar_{g,n} \times \Tor_{\fM_{g,n}} \to
\mbar_{g,n} \times \Tor_{\mbar_{g,n}}  \to  \mbar_{g,n} \times \mbar_{g,n}.
\]
First arrow is saturated and flat, and second arrow strict and smooth.

For the third arrow, note that 
\begin{equation*}
\xymatrix{
\Tor_{\mbar_{g,n}} \ar[r]\ar[d] & \mbar_{g,n} \times \Tor_{\mbar_{g,n}} \ar[d]\\
\mbar_{g,n}\ar[r]^-\Delta          & \mbar_{g,n} \times \mbar_{g,n}\\
}
\end{equation*} 
is cartesian in $\fs$ as well as in $\St$, and horizontal 
arrows are compatible l.c.i.\ maps.   It is then easy to see $\phi'$ and $\Delta$ are compatible.

\subsection{Product formula for families}
As we need the product formula for equivariant invariants in \cite{LLQW}, we will adapt the arguments
above to families. 

Let $X \to S$ and $Y\to T$ be two families of log smooth projective varieties over log smooth and irreducible bases. We would like to relate GW invariants of the family $X\times Y \to S \times T$ to those of $X\to S$ and $Y\to T$.

Denote by $\sM_{g,n}(X/S), \sM_{g,n}(Y/T),\sM_{g,n}(X\times Y/S\times T)$ log stacks of stable log maps from genus $g$ curves with $n$ marked points to these families. 

\begin{remark}
Let $X \to S$ be a family of log smooth projective varieties over a log stack $S$. 
As a cartesian digram 
\[\xymatrix{
X' \ar[r]\ar[d]  & X \ar[d]\\
S' \ar[r]          & S
}\] in $\fs$
induces a cartesian diagram
\[
\xymatrix{
\sM_{g,n}(X'/S') \ar[r]\ar[d] & \sM_{g,n}(X/S)\ar[d]\\
S' \ar[r]          & S
}
\] in $\fs$.
It follows that
$\sM_{g,n}(X/S)$ is an algebraic stack locally of finite type over $S$ by \cite[Theorem 0.1]{GS}.

\end{remark}

It is easy to see we have a cartesian diagram in $\fs$
\[
\xymatrix{
\sM_{g,n}(X\times Y/S\times T)  \ar[r]\ar[d] 
   & \sM_{g,n}(X/S) \times \sM_{g,n}(Y/T)  \ar[d]\\
\fD \ar[r] &\fM_{g,n} \times \fM_{g,n}\ ,
}
\]
and we consider the diagram 
\[
\xymatrix{
\sM_{g,n}(X\times Y/S\times T)  \ar[r]^-h \ar[d]^c 
  & P\ar[r]\ar[d]  
    &   \sM_{g,n}(X/S) \times \sM_{g,n}(Y/T) \ar[d]^a\\
\fD \ar[r]^l\ar[d]^e                                   &  \fP \ar[r]^-\phi \ar[d]& \fM_{g,n} \times \fM_{g,n} \ar[d] \\
\fM_{g,n} &   \mbar_{g,n} \ar[r]^-\Delta & \mbar_{g,n}  \times \mbar_{g,n} \, .
}
\] as before.

\begin{thm}
Let $X \to S$ be 
a family of smooth projective varieties over a smooth, pure dimensional stack $S$, 
and $Y \to T$ a family of log smooth projective varieties over a stack $T$ which is log smooth and irreducible.

Then we have \[
 h_*([\sM_{g,n}(X \times Y/S\times T)]^{\vir})
=\Delta^!([\sM_{g,n}(X/S)]^{\vir} \times [\sM_{g,n}(Y/T)]^{\vir})
\]
\end{thm}

\begin{proof}
Consider the cartesian diagram in $\fs$,
\[
\xymatrix{
\sM_{g,n}(X\times Y/S\times T)  \ar[r]\ar[d] 
   & \sM_{g,n}(X/S) \times \sM_{g,n}(Y/T)  \ar[d]\\
\fD\times (S\times T) \ar[r]\ar[d] &(\fM_{g,n} \times S) \times (\fM_{g,n} \times T) \ar[d]\\
\fD \ar[r] &\fM_{g,n} \times \fM_{g,n}\ .
}
\]
As $T$ is log smooth,  
$\Tor_{\fM_{g,n} \times T} \to \Tor_{\fM_{g,n}}$ is smooth, the perfect obstruction theory for
$\sM_{g,n}(Y/T) \to \fM_{g,n} \times T$ then induces a perfect obstruction theory for the composition
\[
\sM_{g,n}(Y/T) \to \fM_{g,n} \times T \to \fM_{g,n}
\] by the arguments in \cite[Appendix B]{GP}.
Similarly, we have a perfect obstruction theory for 
\[
\sM_{g,n}(X/S) \to \fM_{g,n}\times S \to \fM_{g,n},
\]
and 
\[
\sM_{g,n}(X\times Y/S\times T) \to \fD\times (S\times T) \to \fD
\]
Now we are in the same situation as in Section \ref{QQQ}, 
identical arguments finish the proof. 
Note that (II) follows from standard properties of cotangent complexes of algebraic stacks.
\end{proof}

\section{Remarks for the general case} \label{s3}

\subsection{}
When $V$ has nontrivial log structure, we might consider factoring $a$ by  
\[
\sM_{g,n}(V) \times
\sM_{g,n}(W)  \to \Tor_{\fM_{g,n}} \times \Tor_{\fM_{g,n}}.
\]
In this case, the underlying diagram of  
\begin{equation*}
\xymatrix{
\Tor_{\mbar_{g,n}} \times_{\mbar_{g,n}} \Tor_{\mbar_{g,n}} \ar[r]\ar[d] &
 \Tor_{\mbar_{g,n}} \times \Tor_{\mbar_{g,n}}\ar[d]\\
\mbar_{g,n}  \ar[r]^-\Delta    & \mbar_{g,n} \times \mbar_{g,n}
}
\end{equation*} 
is no longer cartesian in $\St$.
If it is true that $\sM_{g,n}(V) \to \mathfrak{M}_{g,n}$ is saturated, 
we can replace $\Tor_{\fM_{g,n}}$ by its open substack parameterizing  saturated 
maps and  the argument we used can be adapted to this more general setting.

However, the map $\sM_{g,n}(V) \to \mathfrak{M}_{g,n}$ is not even integral in general, and 
further understanding concerning the log structure of $\sM_{g,n}(V)$ seems necessary.

\subsection{}
To prove the product formula holds for $V$ and $W$, we can assume $V$ and $W$ are smooth and log smooth.
This is achieved by using desingularizations of log smooth schemes and invariance of virtual classes under certain log modifications.

By \cite[Theorem 5.10]{Ni}, for any log smooth variety $X$, there exists a log blow up $\pi\colon Y \to X$
such that $Y$ is smooth, log smooth and $\pi$ is birational. In particular, $\pi$ is proper, birational, and log \'etale.

\begin{lemma}
Let $\Phi\colon \widetilde{V} \to V$ and $\Psi\colon \widetilde{W} \to W$ be proper, birational, log \'etale maps between log smooth projective varities.
If the product formula holds for $\widetilde{V}$ and $\widetilde{W}$,
then it holds for $V$ and $W$.

\begin{proof}
Consider the diagram

\[
\xymatrix{
\sM_{g,n}(\widetilde{V} \times \widetilde{W}) \ar@/^/[rrd]^-{\widetilde{h}} \ar@/_/[rdd]_{\sM(\Phi\times \Psi)}\\
&       
   & \widetilde{P}\ar[r]\ar[d]  
      &  \sM_{g,n}(\widetilde{V}) \times \sM_{g,n}(\widetilde{W}) \ar[d]^{\sM(\Phi)\times\sM(\Psi)}\\
&\sM_{g,n}(V \times W)  \ar[r]^-h 
   & P\ar[r]\ar[d]  &  \sM_{g,n}(V) \times \sM_{g,n}(W) \ar[d] \\
&
   &   \mbar_{g,n} \ar[r]^-\Delta 
      & \mbar_{g,n}  \times \mbar_{g,n} \, ,
}
\]
where
\[
\sM(\Phi)\colon \sM_{g,n}(\widetilde{V}) \to \sM(V),
\]
\[
\sM(\Psi)\colon \sM_{g,n}(\widetilde{W}) \to \sM(W),
\]
and 
\[
\sM(\Psi\times \Psi)\colon \sM_{g,n}(\widetilde{V}\times \widetilde{W}) \to \sM(V\times W)
\]
are maps between moduli stacks induced by $\Phi,\Psi$ and $\Phi\times\Psi$ respectively.

As $\Phi$ is proper, birational, log \'etale, by \cite[Theorem 1.1.1]{AW},
\[
\sM(\Phi)_*[\sM_{g,n}(\widetilde{V})]^\vir =[\sM_{g,n}(V)]^\vir.
\]
Similarly, we have
\[
\sM(\Psi)_*[\sM_{g,n}(\widetilde{W})]^\vir =[\sM_{g,n}(W)]^\vir,
\]and
\[
\sM(\Psi\times \Psi)_*[\sM_{g,n}(\widetilde{V}\times \widetilde{W})]^\vir= [\sM(V\times W)]^\vir.
\]

Now 
pushforwarding the relation
\[
\widetilde{h}_*[\sM_{g,n}(\widetilde{V}\times \widetilde{W})]^\vir =
\Delta^!([\sM_{g,n}(\widetilde{V})]^\vir \times [\sM_{g,n}(\widetilde{W})]^\vir)
\] along $\widetilde{P} \to P$ gives
\[
 h_*([\sM_{g,n}(V \times W)]^{\vir})
=\Delta^!([\sM_{g,n}(V)]^{\vir} \times [\sM_{g,n}(W)]^{\vir}).
\]
\end{proof}
\end{lemma}

\begin{remark}
Let $V$ be a smooth projective variety with log structure coming from a simple normal crossing divisor $\cup D_i$.
Motivated by the results proved in \cite{ACW, AMW}, we expect that  genus zero log GW invariants of $V$ are related to orbifold GW invariants of root stacks $V(\sqrt[r]{D_i})$.
If this naive expectation is valid, then the product formula for orbifolds proved in \cite{AJT} would imply the product formula of log GW invariants in genus zero.
\end{remark}

\section{Applications to relative Gromov--Witten invariants} \label{s:3}

We apply this to the \emph{relative GW} invariants as defined by
A.~Li and Y.~Ruan \cite{LR}, and J.~Li \cite{jL}.

Let $X$ and $Y$ be nonsingular projective varieties, 
and $D$ a a smooth divisor in $Y$. 
We further assume $H^1(Y)=0$, so a curve class of $X\times Y$ is of the form $(\beta_X, \beta_Y)$ 
where $\beta_X$ (resp.\ $\beta_Y$) is a curve class of $X$ (resp.\ $Y$).

Let $\mbar_{\Gamma_Y}(Y,D)$ be the relative moduli stack.
Here $\Gamma_Y=(g,n,\beta_Y,\rho, \mu)$ encodes the discrete data:
$g$ for the genus, $n+\rho$ for the number of marked points, $\beta_Y$ the curve class, 
$\mu= (\mu_1, ...\mu_{\rho})$ an ordered partition of $\int_{\beta_Y} [D]$.

We have evaluation maps
\[
 ev_Y \colon  \mbar_{\Gamma_Y}(Y,D) \to Y^n, \quad 
 ev_D\colon \mbar_{\Gamma_Y}(Y,D) \to D^{\rho}
\]
and the stabilization map 
\[
 \pi\colon \mbar_{\Gamma_Y}(Y,D)  \to \mbar_{g,n+\rho}.
\]
Relative GW invariants can be viewed as the 
\emph{Gromov--Witten transformation}
\[
R_{\Gamma_Y}\colon  H^*(Y)^{\otimes n} \otimes H^*(D)^{\otimes \rho} \to H^*(\mbar_{g,n+\rho})
\]
defined as 
\[
  \operatorname{PD} \left( \pi_* \left( ev_Y^*(\alpha ) ev_D^*( \delta ) 
  \cap [\mbar_{\Gamma_Y}(Y,D)]^{vir} \right) \right),
\]
where $\operatorname{PD}$ stands for the Poincar\'e duality.

Let $\Gamma_{X \times Y} =(g,n,(\beta_X,\beta_Y),\rho,\mu)$. 
Similarly we have 
\[
R_{\Gamma_{X \times Y}}\colon 
H^*(X \times Y)^{\otimes n} \otimes H^*(X \times D)^{\otimes \rho} \to H^*(\mbar_{g,n+\rho}).
\]

Let $\Gamma_X=(g,n+\rho, \beta_X)$.
The map 
\[
 R_{\Gamma_X}\colon  H^*(X)^{\otimes (n+\rho) }  \to H^*(\mbar_{g,n+\rho})
\] 
is the Gromov--Witten correspondence $R^X_{g,n+\rho,\beta_X}$,
defining a cohomological field theory.


\begin{corollary} \label{c:2.2}
\[
 \begin{split}
  &R_{\Gamma_{X \times Y}}
   ((\alpha_1 \otimes \alpha'_1) \otimes ... \otimes  (
  (\alpha_n \otimes \alpha'_n));
   (\alpha_{n+1} \otimes \delta_1) \otimes ... \otimes ((\alpha_{n+\rho} \otimes \delta_\rho)) \\
 = &R_{\Gamma_X} (\alpha_1 \otimes ... \otimes  \alpha_{n+\rho}) 
   R_{\Gamma_Y} (\alpha'_1 \otimes ... \otimes \alpha'_n; 
   \delta_1\otimes ...\otimes \delta_\rho),
 \end{split}
\]
where $\alpha_i \in H^*(X)$, $\alpha'_i \in H^*(Y)$ and
$\delta_j \in H^*(D)$.
\end{corollary}
\begin{proof}
This follows directly from Theorem~\ref{t:main} and the comparison result 
between relative and log GW invariants in \cite[Theorem 1.1, Section 2.3]{AMW}.
\end{proof}

\subsection*{Acknowledgements}
We wish to thank Q.~Chen, W.~D.~Gillam, and S.~Marcus for helpful discussions, and the referee for valuable comments. 
This research is supported in part by the NSF.

\end{document}